\documentclass[12pt, reqno]{amsart}
\usepackage{amsmath, amsthm, amscd, amsfonts, amssymb, graphicx, color}
\usepackage[bookmarksnumbered, colorlinks, plainpages]{hyperref}
\hypersetup{colorlinks=true,linkcolor=red, anchorcolor=green, citecolor=cyan, urlcolor=red, filecolor=magenta, pdftoolbar=true}
\input{mathrsfs.sty}
\usepackage{xcolor}
\usepackage{xcolor}
\usepackage{ulem}

\newtheorem{theorem}{Theorem}[section]
\newtheorem{lemma}[theorem]{Lemma}
\newtheorem{proposition}[theorem]{Proposition}
\newtheorem{cor}[theorem]{Corollary}
\theoremstyle{definition}

\newtheorem{example}[theorem]{Example}

\theoremstyle{remark}
\newtheorem{remark}[theorem]{\bf{Remark}}
\numberwithin{equation}{section}
\begin{document}
		\title [Inequalities involving Berezin norm and  Berezin number]{Inequalities involving Berezin norm and  Berezin number} 
	
	\author[ P. Bhunia, K. Paul, A. Sen] {Pintu Bhunia, Kallol Paul, Anirban Sen}

	\address [Bhunia] {Department of Mathematics, Jadavpur University, Kolkata 700032, West     Bengal, India}
	\email{pintubhunia5206@gmail.com}
	\email{pbhunia.math.rs@jadavpuruniversity.in}

	\address[Paul] {Department of Mathematics, Jadavpur University, Kolkata 700032, West Bengal, India}
	\email{kalloldada@gmail.com}
	\email{kallol.paul@jadavpuruniversity.in}
	
		\address [Sen] {Department of Mathematics, Jadavpur University, Kolkata 700032, West Bengal, India}
	\email{anirbansenfulia@gmail.com}
	
	\thanks{Mr. Pintu Bhunia sincerely acknowledges the financial support received from UGC, Govt. of India in the form of Senior Research Fellowship under the mentorship of Prof Kallol Paul. Mr. Anirban Sen would like to thank CSIR, Govt. of India for the financial support in the form of Junior Research Fellowship under the mentorship of Prof Kallol Paul. }
	
	
	
	\maketitle
	\begin{abstract}
		We obtain  new inequalities involving Berezin norm and Berezin number of bounded linear operators defined on a reproducing kernel Hilbert space $\mathscr{H}.$  Among many  inequalities obtained here, it is shown that if $A$ is a positive bounded linear operator on $\mathscr{H}$, then $\|A\|_{ber}=\textbf{ber}(A)$, where $\|A\|_{ber}$ and $\textbf{ber}(A)$ are the Berezin norm and Berezin number of $A$, respectively. In contrast to the numerical radius, this equality does not hold  for selfadjoint operators, which highlights the necessity of studying Berezin number inequalities independently. 
		
	\end{abstract}

	\section{\textbf{Introduction}}
	\noindent 
	Let $\mathbb{B}(\mathbb{H})$ denote the $C^*$-algebra of all bounded linear operators on a complex Hilbert space $\mathbb{H}$ with the usual inner product $\langle.,.\rangle$, and $\|\cdot \|$ is the norm induced by the inner product $\langle.,.\rangle$. The alphabet $I$ stands for the identity operator on $\mathbb{H}$. 
	An operator $A \in \mathbb{B}(\mathbb{H})$ is positive if and only if $\langle Ax,x \rangle \geq 0$ for all $x \in \mathbb{H}$, and we write $A \geq 0.$   Let $A\in \mathbb{B}(\mathbb{H}).$ The adjoint  of $A$ is denoted by $A^*$ and  $|A|$ denotes the positive operator $(A^*A)^{1/2}.$
	The Cartesian decomposition of $A$ is given by $A=\Re(A)+ { i }\Im(A),$ where $\Re(A)$ and $\Im(A)$ denote the real part and the imaginary part of $A$, respectively, i.e., $ \Re(A) = \frac{ A + A^*}{2} $ and $ \Im(A) = \frac{A - A^*}{2i}.$
	For $A\in \mathbb{B}(\mathbb{H})$, the numerical range of $A$, denoted as $W(A)$, is the collection of complex scalars $\langle Ax,x \rangle$ for $x\in \mathbb{H}$ with $\|x\|=1$. More precisely, 
	$$ W(A)=\left\{ \langle Ax,x \rangle :  x \in \mathbb{H}, \|x\|=1\right\}.$$ 
	The numerical radius and the usual operator norm of $A$ are denoted by $w(A)$ and $\|A\|$, respectively. Recall that
\begin{eqnarray*}
		 w(A) &=& \sup \left\{\big |\langle Ax,x \rangle\big| :  x \in \mathbb{H}, \|x\|=1\right\}
\end{eqnarray*}	  
and
\begin{eqnarray*}	
\|A\| &=& \sup \left\{\big|\langle Ax,y \rangle\big| :  x,y \in \mathbb{H}, \|x\|= \|y\|=1 \right\}.
\end{eqnarray*}	
It is easy to verify that $w(.)$ defines a norm on  $\mathbb{B}(\mathbb{H})$ and is equivalent to the usual operator norm. In particular, for all  $A \in  \mathbb{B}(\mathbb{H})$, the following inequality holds 
\begin{eqnarray}\label{eq0}
	\frac{1}{2}\|A\| \leq w(A) \leq \|A\|. 
\end{eqnarray}
For the latest and recent improvements of the above inequalities in \eqref{eq0} one can see \cite{B_RM, B_LAA, B_BM, B_AM} and references therein.

A reproducing kernel Hilbert space (RKHS in short) $\mathscr{H}=\mathscr{H}(\Omega)$ is a Hilbert space of all complex valued functions on a non-empty set $\Omega,$ which has the property that for every $\lambda \in \Omega$ the map $ E_{\lambda} : \mathscr{H} \to \mathbb{C}$  defined by $E_{\lambda}(f)=f(\lambda),$ is continuous linear functional on $\mathscr{H}.$  Throughout the article, a reproducing kernel Hilbert space on the set $\Omega$ is denoted by $\mathscr{H}.$ By the Riesz representation theorem, for each $\lambda \in \Omega$ there exists a unique function $k_{\lambda} \in \mathscr{H}$ such that $f(\lambda)=\langle f,k_{\lambda} \rangle$ for all $f \in \mathscr{H}.$ The collection of functions  $\{k_\lambda :  \lambda \in \Omega \}$ is called  the reproducing kernel of $\mathscr{H}.$ The normalized reproducing kernel of $\mathscr{H}$ is the collection of functions $\{\hat{k}_{\lambda}=k_\lambda/\|k_\lambda\| :  \lambda \in \Omega\}.$ Let  $A \in  \mathbb{B}(\mathscr{H}).$ The function $\widetilde{A}$ defined on $\Omega$ by $\widetilde{A}(\lambda)=\langle A\hat{k}_{\lambda},\hat{k}_{\lambda} \rangle$,  is called the Berezin symbol of $A.$ The Berezin set of $A$, denoted as $\textbf{Ber}(A),$ is defined as $\textbf{Ber}(A)=\{\widetilde{A}(\lambda) : \lambda \in \Omega\}$. 
The Berezin number and Berezin norm of $A$ denoted as $\textbf{ber}(A)$ and $\|A\|_{ber},$ respectively are defined as 
	\begin{eqnarray*}
		  \textbf{ber}(A) &=& \sup \left\{ \big| \widetilde{A}(\lambda) \big| : \lambda \in \Omega\right\}=\sup \left\{ \big| \langle A\hat{k}_{\lambda},\hat{k}_{\lambda} \rangle \big| : \lambda \in \Omega\right\}
		\end{eqnarray*}
	and
	\begin{eqnarray*}
		\|A\|_{ber}&=&\sup \left\{ \big|\langle A\hat{k}_{\lambda},\hat{k}_{\mu} \rangle\big| : \lambda, \mu \in \Omega\right\}.
	\end{eqnarray*}
For $A,B \in  \mathbb{B}(\mathscr{H})$ it is clear from the definition of the Berezin number and the Berezin norm that the following properties hold: 
	\begin{eqnarray*}
		&&(i)~~\textbf{ber}(\alpha A) = | \alpha |  \textbf{ber}(A)~~ \mbox{for all}~~  \alpha  \in \mathbb{C},\\
		&&(ii)~~\textbf{ber}(A + B) \leq \textbf{ber}(A)  + \textbf{ber}(B),\\&&(iii)~~\textbf{ber}(A) \leq \|A\|_{ber},\\&&(iv)~~ \|\alpha A\|_{ber}=|\alpha|\|A\|_{ber}~~ \mbox{for all}~~  \alpha  \in \mathbb{C},\\&&(v)~~ \|A+B\|_{ber} \leq \|A\|_{ber}+\|B\|_{ber},\\&& (vi)~~\|A\|_{ber}=\|A^*\|_{ber}\,\,\, \text{and} \,\,\, \textbf{ber}(A)=\textbf{ber}(A^*).
 \end{eqnarray*}	
Also, it is clear that  for $A\in \mathbb{B}(\mathscr{H}),$ $$ \textbf{Ber}(A) \subseteq W(A), \textbf{ber}(A) \leq w(A) ~~\mbox{and}~~\|A\|_{ber} \leq \|A\|.$$ 
The Berezin symbol has been studied in details for Toeplitz and Hankel operators on Hardy and Bergman spaces. The Berezin number inequalities have been studied by many mathematicians over the years, the interested readers can see \cite{B1,b0,G1,hlb,Y1}.

In this paper, we obtain generalized inequalities involving Berezin norm and Berezin number of reproducing kernel Hilbert space operators.	As special cases, we derive several inequalities that refine the existing ones. Further, we also give some usual operator norm inequalities of complex Hilbert space operators.

\section{\textbf{The Berezin norm inequalities}}

We start with the following lemmas that will be used to develop new results in this article. 
	
		\begin{lemma}$($\cite{Simon}$)$\label{lemma1}
		Let $A\in \mathbb{B}(\mathscr{H})$ be positive, and let $x\in \mathscr{H}$ with $\|x\|=1.$ Then 
		\begin{eqnarray*}
			\langle Ax,x\rangle^r\leq \langle A^rx,x\rangle~~\,\, \text{for all}~~ r\geq 1.
		\end{eqnarray*}
	\end{lemma}

    \begin{lemma}$($\cite{F}$)$\label{lemma2}
    	Let $A\in \mathbb{B}(\mathscr{H})$, and let $x, y\in \mathscr{H}$. Then 
	   \[|\langle Ax,y\rangle|^2\leq \langle |A|^{2\alpha}x,x\rangle \langle |A^*|^{2(1-\alpha)}y,y\rangle \,\,\,\, \text{for all  $\alpha \in [0, 1]$}.\]
   \end{lemma}

    \begin{lemma}$($\cite{hard}$)$.\label{lemma3}
   	For $a,b\geq 0,$ $0<\alpha<1$ and $r\neq 0,$ let $M_{r}(a,b,\alpha)=\\\left(\alpha a^r+(1-\alpha)b^r \right)^{1/r}$ and $M_{0}(a,b,\alpha)=a^{\alpha}b^{1-\alpha}.$ Then 
   	$$M_{r}(a,b,\alpha)\leq M_{s}(a,b,\alpha) \,\,\,\,\text{for $r \leq s$}. $$ 
   \end{lemma}

Now, we are in a position to prove a general inequality involving the Berezin norm and Berezin number, which leads to several inequalities as special cases.

\begin{theorem}\label{theo1}
	Let $A,B,C,D,X,Y\in \mathbb{B}(\mathscr{H})$, and let $\alpha \in [0,1].$ Then
	\begin{align*}
		&\left\|\frac{A^*XB+C^*YD}{2}\right\|^{2}_{ber} \\& \leq  \textbf{ber}^{1/r}\left(\frac{(B^*|X|^{2\alpha }B)^r+(D^*|Y|^{2\alpha }D)^r}{2}\right)\textbf{ber}^{1/s}\left(\frac{(A^*|X^*|^{2(1-\alpha)}A)^s +(C^*|Y^*|^{2(1-\alpha)}C)^s}{2}\right),
	\end{align*}
	for all $r,s \geq 1.$ 
\end{theorem}

\begin{proof}
	Let $\hat{k}_{\lambda}$ and $\hat{k}_{\mu}$ be two normalized reproducing kernel of $\mathscr{H}.$ Then, 
	\begin{align*}
		& \frac{1}{4}|\langle (A^*XB+C^*YD)\hat{k}_{\lambda},\hat{k}_{\mu} \rangle |^2\\
		&\leq \frac{1}{4}(|\langle A^*XB\hat{k}_{\lambda},\hat{k}_{\mu} \rangle |+|\langle C^*YD \hat{k}_{\lambda},\hat{k}_{\mu} \rangle |)^2 \\
		&= \frac{1}{4}(|\langle XB\hat{k}_{\lambda},A\hat{k}_{\mu} \rangle |+|\langle YD \hat{k}_{\lambda},C\hat{k}_{\mu} \rangle |)^2\\
		&\leq \frac{1}{4} \Big(\langle |X|^{2\alpha}B\hat{k}_{\lambda},B\hat{k}_{\lambda}\rangle^{1/2} \langle |X^*|^{2(1-\alpha)}A\hat{k}_{\mu},A\hat{k}_{\mu}\rangle^{1/2}\\
		&\,\,\,\,\,\,\,\,\,\,\,\,+\langle|Y|^{2\alpha}D\hat{k}_{\lambda},D\hat{k}_{\lambda}\rangle^{1/2} \langle |Y^*|^{2(1-\alpha)}C\hat{k}_{\mu},C\hat{k}_{\mu}\rangle^{1/2}\Big)^2\,\,\,\,\,\,\,\, \Big(\mbox{by Lemma \ref{lemma2}}\Big)\\
		&= \frac{1}{4} \Big(\langle B^*|X|^{2\alpha}B\hat{k}_{\lambda},\hat{k}_{\lambda}\rangle^{1/2} \langle A^*|X^*|^{2(1-\alpha)}A\hat{k}_{\mu},\hat{k}_{\mu}\rangle^{1/2}\\
		&\,\,\,\,\,\,\,\,\,\,\,\,\,\,+\langle D^*|Y|^{2\alpha}D\hat{k}_{\lambda},\hat{k}_{\lambda}\rangle^{1/2} \langle C^*|Y^*|^{2(1-\alpha)}C\hat{k}_{\mu},\hat{k}_{\mu}\rangle^{1/2}\Big)^2\\
		&\leq \frac{1}{4} \Big (\langle B^*|X|^{2\alpha}B\hat{k}_{\lambda},\hat{k}_{\lambda}\rangle + \langle  D^*|Y|^{2\alpha}D \hat{k}_{\lambda},\hat{k}_{\lambda}\rangle \Big)\\
		&\,\,\,\,\,\,\,\,\,\,\,\,\, \times \Big(\langle A^*|X^*|^{2(1-\alpha)}A \hat{k}_{\mu},\hat{k}_{\mu}\rangle+\langle C^*|Y^*|^{2(1-\alpha)}C \hat{k}_{\mu},\hat{k}_{\mu}\rangle \Big)\\
		& \Big(\mbox{using the inequality $(ab+cd)^2\leq(a^2+c^2)(b^2+d^2)$ for all $a,b,c,d \in \mathbb{R}$}\Big)\\
		&=  \left(\frac{\langle B^*|X|^{2\alpha}B \hat{k}_{\lambda},\hat{k}_{\lambda}\rangle + \langle D^*|Y|^{2\alpha}D \hat{k}_{\lambda},\hat{k}_{\lambda}\rangle}{2}\right)\\&\,\,\,\,\,\,\,\,\,\,\, \times \left(\frac{\langle A^*|X^*|^{2(1-\alpha)}A \hat{k}_{\mu},\hat{k}_{\mu}\rangle+\langle C^*|Y^*|^{2(1-\alpha)}C \hat{k}_{\mu},\hat{k}_{\mu}\rangle}{2}\right) \\& \leq  \left(\frac{\langle B^*|X|^{2\alpha}B\hat{k}_{\lambda},\hat{k}_{\lambda}\rangle^r + \langle D^*|Y|^{2\alpha}D \hat{k}_{\lambda},\hat{k}_{\lambda}\rangle^r}{2}\right)^{1/r}\\&\,\,\,\,\,\,\,\,\,\, \times \left(\frac{\langle A^*|X^*|^{2(1-\alpha)}A\hat{k}_{\mu},\hat{k}_{\mu}\rangle^s+\langle C^*|Y^*|^{2(1-\alpha)}C\hat{k}_{\mu},\hat{k}_{\mu}\rangle^s}{2}\right)^{1/s}\,\,\,\,\,\,\,\, \Big(\mbox{by Lemma \ref{lemma3}}\Big)\\& \leq  \left(\frac{\langle (B^*|X|^{2\alpha }B)^r\hat{k}_{\lambda},\hat{k}_{\lambda}\rangle + \langle (D^*|Y|^{2\alpha }D)^r\hat{k}_{\lambda},\hat{k}_{\lambda}\rangle}{2}\right)^{1/r}\\&\,\,\,\,\,\,\,\,\, \times \left(\frac{\langle  (A^*|X^*|^{2(1-\alpha)}A)^s \hat{k}_{\mu},\hat{k}_{\mu}\rangle+\langle (C^*|Y^*|^{2(1-\alpha)}C)^s \hat{k}_{\mu},\hat{k}_{\mu}\rangle}{2}\right)^{1/s}\,\,\,\,\,\, \Big(\mbox{by Lemma \ref{lemma1}}\Big)\\&=  \,\left\langle \left(\frac{(B^*|X|^{2\alpha }B)^r+(D^*|Y|^{2\alpha }D)^r}{2}\right) \hat{k}_{\lambda},\hat{k}_{\lambda}\right\rangle^{1/r}\\& \,\,\,\,\,\,\,\,\,\, \times \left \langle \left(\frac{(A^*|X^*|^{2(1-\alpha)}A)^s +(C^*|Y^*|^{2(1-\alpha)}C)^s}{2}\right)\hat{k}_{\mu},\hat{k}_{\mu}\right \rangle^{1/s}\\ &\leq \textbf{ber}^{1/r}\left(\frac{(B^*|X|^{2\alpha }B)^r+(D^*|Y|^{2\alpha }D)^r}{2}\right)\textbf{ber}^{1/s}\left(\frac{(A^*|X^*|^{2(1-\alpha)}A)^s +(C^*|Y^*|^{2(1-\alpha)}C)^s}{2}\right).
	\end{align*}
	Therefore, taking the supremum over all $\lambda,\mu\in\Omega$, we get the desired inequality.

\end{proof}

The following corollary follows easily from Theorem \ref{theo1} by taking $A=B=C=D=I,$ $X=A$ and $Y=B.$

\begin{cor}\label{cor1}
	Let $A,B\in \mathbb{B}(\mathscr{H})$, and let $\alpha \in [0,1].$ Then
	\begin{eqnarray*}
		\left\|\frac{A+B}{2}\right\|^{2}_{ber} \leq \textbf{ber}^{1/r}\left(\frac{|A|^{2\alpha r}+|B|^{2\alpha r}}{2}\right)\textbf{ber}^{1/s}\left(\frac{|A^*|^{2(1-\alpha) s}+|B^*|^{2(1-\alpha) s}}{2}\right),
	\end{eqnarray*}
	for all $r,s \geq 1.$ 

In particular, for $s=r$
\begin{eqnarray}\label{eqn1}
	\left\|A+B\right\|^{r}_{ber} \leq 2^{r-1}\textbf{ber}^{1/2}\left(|A|^{2\alpha r}+|B|^{2\alpha r}\right)\textbf{ber}^{1/2}\left(|A^*|^{2(1-\alpha) r}+|B^*|^{2(1-\alpha)r}\right),
	\end{eqnarray}
	for all $r\geq 1.$ 
\end{cor}

\begin{remark}
	It was proved in \cite[Th. 2.17.]{Bm} that if $A,B\in \mathbb{B}(\mathscr{H})$, then
	\begin{eqnarray}\label{eqn2}
	\|A+B\|^{r}_{ber} \leq  2^{r-2}\left(\textbf{ber}(|A|^{2\alpha r}+|B|^{2\alpha r})+\textbf{ber}(|A^*|^{2(1-\alpha) r}+|B^*|^{2(1-\alpha) r})\right),
	\end{eqnarray}
for $0<\alpha<1$ and for all $r \geq 1.$ Clearly, by applying AM-GM inequality, we infer that 
	\begin{align*}
		&\textbf{ber}^{1/2}\left(|A|^{2\alpha r}+|B|^{2\alpha r}\right)\textbf{ber}^{1/2}\left(|A^*|^{2(1-\alpha) r}+|B^*|^{2(1-\alpha)r}\right)\\ &\leq \frac{1}{2}\left(\textbf{ber}(|A|^{2\alpha r}+|B|^{2\alpha r})+\textbf{ber}(|A^*|^{2(1-\alpha) r}+|B^*|^{2(1-\alpha) r})\right). 
	\end{align*}
Thus, the inequality (\ref{eqn1}) is sharper than the inequality (\ref{eqn2}).
\end{remark}

Again, by considering $X=Y=I$ in Theorem \ref{theo1}, we get the following inequality.

\begin{cor}\label{cor3}
	Let $A,B,C,D\in \mathbb{B}(\mathscr{H}).$ Then
	\begin{align}\label{eq1}
	\left\|\frac{A^*B+C^*D}{2}\right\|^{2}_{ber}  \leq \textbf{ber}^{1/r}\left(\frac{|A|^{2r} +|C|^{2r}}{2}\right) \textbf{ber}^{1/s}\left(\frac{|B|^{2s}+|D|^{2s}}{2}\right),
	\end{align}
	for all $r,s \geq 1.$ 
\end{cor}

\begin{remark}
It was  proved in \cite[Theorem 3.5.]{ceb} that if $A,B,C,D\in \mathbb{B}(\mathscr{H}),$ then
	\begin{align}\label{Ceb}
	\textbf{ber}^2\left(\frac{A^*B+C^*D}{2}\right)  \leq \textbf{ber}^{1/r}\left(\frac{|A|^{2r} +|C|^{2r}}{2}\right) \textbf{ber}^{1/s}\left(\frac{|B|^{2s}+|D|^{2s}}{2}\right),
	\end{align}
for all $r,s \geq 1.$

Since $\textbf{ber} (A^*B+C^*D) \leq \| A^*B+C^*D \|_{ber}$, the inequality \eqref{Ceb} follows from (\ref{eq1}).
\end{remark}

Next inequality follows from Corollary \ref{cor3} by taking $s=r$.

\begin{cor}\label{cor4}
	Let $A,B,C,D\in \mathbb{B}(\mathscr{H}).$  Then 
	\begin{align*}
	\left\|\frac{A^*B+C^*D}{2}\right\|^{2r}_{ber}  \leq \textbf{ber}\left(\frac{|A|^{2r} +|C|^{2r}}{2}\right) \textbf{ber}\left(\frac{|B|^{2r} +|D|^{2r}}{2}\right),
	\end{align*}
for all $r \geq 1.$
\end{cor}

Now, we prove an interesting equality for positive operators.

\begin{proposition}\label{prop1}
	If $A\in \mathbb{B}(\mathscr{H})$ is positive (i.e., $A\geq 0$), then
	\[\left\|A\right\|_{ber}=  \textbf{ber}\left(A\right).\]
\end{proposition}

\begin{proof}
Putting $A=B=X$, $C=D=0$ and $r=1$ in Corollary \ref{cor4}, we have
\begin{align*}
	\left\|X^*X\right\|_{ber}\leq \textbf{ber}\left(X^*X\right) \,\, \text{for every $X\in \mathbb{B}(\mathscr{H})$}.
\end{align*}
Therefore, for every $X\in \mathbb{B}(\mathscr{H})$
\begin{align}\label{eql2}
	\left\|X^*X\right\|_{ber}= \textbf{ber}\left(X^*X\right).
\end{align}
Since $A$ is positive, there exists a $Y\in \mathbb{B}(\mathscr{H})$ such that $A=Y^*Y$. This argument together with \eqref{eql2} gives that $	\left\|A\right\|_{ber}=  \textbf{ber}\left(A\right).$
\end{proof}

In the following example we show that the above proposition may not be true for  selfadjoint operators.

\begin{example}
	Considering $\mathbb{C}^{2n}$ as a RKHS (see in \cite[pp. 4-5]{book}), if we take 
	$$A=   \begin{pmatrix}
0&0&\ldots&0&1\\
0&0&\ldots&1&0\\
\vdots\\
0&1&\ldots&0&0\\
1&0&\ldots&0&0
\end{pmatrix}_{2n \times 2n},$$ then $\|A\|_{ber}=1\neq 0= \textbf{ber}(A)$. 
	It is easy to verify that  $A$ is a selfadjoint operator on $\mathbb{C}^{2n}$ and so  $w(A)=\|A\|$ The difference between this two relations is due to the fact that the collection of normalized reproducing kernel for $\mathbb{C}^{2n}$ is precisely the $2n$ elements $e_1,e_2,\ldots, e_{2n}$ ($e_i(j)=\delta_{ij}$, $\delta_{ij}$ is kronecker delta function), whereas the unit sphere of $\mathbb{C}^{2n}$ consist of uncountably many elements. This justifies the study of Berezin norm and Berezin number inequalities independently.


\end{example}

Next inequality follows by taking $C=B$ and $D=A$ in Corollary \ref{cor3}.

\begin{cor}\label{cor5}
	Let $A,B\in \mathbb{B}(\mathscr{H}).$ Then
	\begin{align*}
	&(i)~~\left\|\frac{A^*B+B^*A}{2}\right\|^{2}_{ber}  \leq \textbf{ber}^{1/r}\left(\frac{|A|^{2r} +|B|^{2r}}{2}\right) \textbf{ber}^{1/s}\left(\frac{|A|^{2s}+|B|^{2s}}{2}\right),\\&~~\mbox{for all $r,s \geq 1$}.\\
	&(ii)~~\left\|\frac{A^*B+B^*A}{2}\right\|^{r}_{ber}  \leq \textbf{ber}\left(\frac{|A|^{2r} +|B|^{2r}}{2}\right),\\&~~\mbox{for all $r \geq 1.$}
	\end{align*} 

	In particular, for $r=1$
	\begin{eqnarray}
		\|A^*B+B^*A\|_{ber} \leq \textbf{ber}(A^*A+B^*B) = \|A^*A+B^*B\|_{ber}.
	\end{eqnarray} 

\end{cor}

Now, taking $A=C=I, B=A$ and $D=B$ in Corollary \ref{cor3} we get the following corollary.

\begin{cor}\label{cor11}
	Let $A,B\in \mathbb{B}(\mathscr{H}).$ Then
	\begin{align}\label{eqn21}
	\left\|\frac{A+B}{2}\right\|^{2r}_{ber} \leq \textbf{ber}\left(\frac{|A|^{2r} +|B|^{2r}}{2}\right),
	\end{align}
	 for all $r \geq 1.$
	 
	 In particular, for $r=1$ 
	 \begin{align}
	    \left\|A+B\right\|^{2}_{ber} \leq 2\,\textbf{ber}(A^*A+B^*B) = 2\|A^*A+B^*B\|_{ber}
	 \end{align}
 
\end{cor}

	If we take $A=\Re(A)$ and $B=\rm{i}\Im(A)$ in (\ref{eqn21}), then we get 
	  \begin{align}
	 \|A\|^{2r}_{ber} \leq 2^{2r-1}\textbf{ber}\left(\Re(A)^{2r}+\Im(A)^{2r}\right),
	  \end{align}
	  for all $r \geq 1.$
	  
	  Also, if we take $A=A$ and $B=A^*$ in (\ref{eqn21}), then we get
	   \begin{align}
	  \|\Re(A)\|^{2r}_{ber} \leq \textbf{ber}\left(\frac{|A|^{2r}+|A^*|^{2r}}{2}\right) \,\, \text{for all $r \geq 1$}
	  \end{align}
	 and if we take $A=A$ and $B=-A^*$ in (\ref{eqn21}), then we get
	   \begin{align}
	   \|\Im(A)\|^{2r}_{ber} \leq \textbf{ber}\left(\frac{|A|^{2r}+|A^*|^{2r}}{2}\right) \,\,\text{for all $r \geq 1.$}
	   \end{align}

The following result follows from Corollary \ref{cor3} by considering $A=A^*, B=A, C=B^*$ and $D=B.$

\begin{cor}\label{cor6}
		Let $A,B\in \mathbb{B}(\mathscr{H}).$ Then
	\begin{align*}
	  	&(i)~~\left\|\frac{A^2+B^2}{2}\right\|^{2}_{ber} \leq \textbf{ber}^{1/r}\left(\frac{|A|^{2r} +|B|^{2r}}{2}\right) \textbf{ber}^{1/s}\left(\frac{|A^*|^{2s} +|B^*|^{2s}}{2}\right),\\& ~~\mbox{for all $r,s \geq 1.$}\\
	  	&(ii)~~\left\|\frac{A^2+B^2}{2}\right\|^{2r}_{ber} \leq \textbf{ber}\left(\frac{|A|^{2r} +|B|^{2r}}{2}\right) \textbf{ber}\left(\frac{|A^*|^{2r} +|B^*|^{2r}}{2}\right),\\& ~~\mbox{for all $r \geq 1.$}
	\end{align*}

	In particular, for $r=1$ 
	\begin{align}
	\left\|A^2+B^2\right\|^{2}_{ber} & \leq \textbf{ber}(A^*A+B^*B)\textbf{ber}(AA^*+BB^*) \\
	&= \|A^*A+B^*B\|_{ber}\|AA^*+BB^*\|_{ber} \nonumber.
	\end{align}
\end{cor}

Also, the following inequality follows from Corollary \ref{cor3} by choosing $A=I, D=I, C=B^*$ and $B=A$.

\begin{cor}\label{cor7}
	Let $A,B\in \mathbb{B}(\mathscr{H}).$ Then 
	\begin{align}\label{eqn3}
	     \left\|\frac{A+B}{2}\right\|^{2}_{ber} \leq \textbf{ber}^{1/r}\left(\frac{|A|^{2r} +I}{2}\right) \textbf{ber}^{1/s}\left(\frac{|B^*|^{2s} +I}{2}\right),
	\end{align}
 for all $r,s \geq 1.$
 
	In particular, for $B=A$ 
	\begin{align}\label{eqn4}
	\|A\|^{2}_{ber} \leq \textbf{ber}^{1/r}\left(\frac{|A|^{2r} +I}{2}\right) \textbf{ber}^{1/s}\left(\frac{|A^*|^{2s} +I}{2}\right),
	\end{align}
	for all $r,s \geq 1.$
	
	 Moreover, for $s=r$
	\begin{align}\label{eqn5}
	\|A\|^{2r}_{ber} \leq \textbf{ber}\left(\frac{|A|^{2r} +I}{2}\right)\textbf{ber}\left(\frac{|A^*|^{2r} +I}{2}\right),
	\end{align}
 for all $r \geq 1.$
\end{cor}

Now, taking $A=A^*$ and $C=D=0$ in Corollary \ref{cor3} we get the following result.

\begin{cor}
		Let $A,B\in \mathbb{B}(\mathscr{H}).$ Then
		\begin{align*}
		      	&(i)~~\left\|AB \right\|^{2}_{ber} \leq 2^{2-1/r-1/s}\textbf{ber}^{1/r}\left(|A^*|^{2r}\right) \textbf{ber}^{1/s}\left(|B|^{2s}\right)~~\mbox{for all $r,s \geq 1.$}\\
		      	&(ii)~~\left\|AB \right\|^{2r}_{ber} \leq 2^{2r-2}\textbf{ber}\left(|A^*|^{2r}\right) \textbf{ber}\left(|B|^{2r}\right)~~\mbox{for all $r \geq 1.$}
		\end{align*}
		In particular, for $r=1$
		\begin{eqnarray}
			\left\|AB \right\|_{ber} \leq {\textbf{ber}^{1/2}\left(AA^*\right) \textbf{ber}^{1/2}\left(B^*B\right)}=\left\|AA^* \right\|^{1/2}_{ber}\left\|B^*B \right\|^{1/2}_{ber}.
		\end{eqnarray}
\end{cor}

Also, the next result follows from Corollary \ref{cor3} by taking $A=A^*,B=B,C=\pm B^*$ and $D=A.$

\begin{cor}\label{cor8}
		Let $A,B\in \mathbb{B}(\mathscr{H}).$ Then  
		\begin{align*}
		&(i)~~\left\|\frac{AB \pm BA}{2}\right\|^{2}_{ber} \leq \textbf{ber}^{1/r}\left(\frac{|A|^{2r} +|B|^{2r}}{2}\right) \textbf{ber}^{1/s}\left(\frac{|A^*|^{2s}+|B^*|^{2s}}{2}\right),\\&~~\mbox{for all $r,s \geq 1.$}\\
		&(ii)~~\left\|\frac{AB \pm BA}{2}\right\|^{2r}_{ber} \leq \textbf{ber}\left(\frac{|A|^{2r} +|B|^{2r}}{2}\right) \textbf{ber}\left(\frac{|A^*|^{2r}+|B^*|^{2r}}{2}\right),\\&~~\mbox{for all $r\geq 1.$}
		\end{align*}
	\end{cor}
	
	In particular, for $B=A^*$ in Corollary \ref{cor8} (ii), we have
		\begin{align}\label{eqn6}
		\left\|AA^*\pm A^*A\right\|^{r}_{ber} \leq 2^{r-1}\textbf{ber}\left((A^*A)^r +(AA^*)^r\right)\,\, \text{for all $r \geq 1.$}
		\end{align}
	
		Moreover, for $r=1$ in (\ref{eqn6}), we have
		\begin{align}\label{eql1}
		   \left\|AA^*- A^*A\right\|_{ber} \leq \textbf{ber}\left(A^*A +AA^*\right) = \left\|AA^*+A^*A\right\|_{ber}.
		\end{align}



In the following theorem we obtain an upper bound for the  Berezin norm of  the sum of the product of two positive operators.

\begin{theorem}\label{theo2}
		Let $A,B\in \mathbb{B}(\mathscr{H})$ be positive. Then 
		\begin{eqnarray*}
			\left\|\frac{A^{\alpha}B^{1-\alpha}+A^{1-\alpha}B^{\alpha}}{2}\right\|^{2}_{ber} \leq \textbf{ber}^{1/r}\left(\frac{A^{2\alpha r} +A^{2(1-\alpha)r}}{2}\right) \textbf{ber}^{1/s}\left(\frac{B^{2\alpha s} +B^{2(1-\alpha)s}}{2}\right),
    	\end{eqnarray*}
		 for all $r,s \geq 1$ and for all $\alpha \in [0,1].$
\end{theorem}

\begin{proof}
		Suppose that $\hat{k}_{\lambda}$ and $\hat{k}_{\mu}$ are two normalized reproducing kernel of $\mathscr{H}.$ Then
		\begin{eqnarray*}
			&&|\langle (A^{\alpha}B^{1-\alpha}+A^{1-\alpha}B^{\alpha})\hat{k}_{\lambda},\hat{k}_{\mu} \rangle|^2\\& \leq & (|\langle A^{\alpha}B^{1-\alpha} \hat{k}_{\lambda},\hat{k}_{\mu} \rangle|+|\langle A^{1-\alpha}B^{\alpha} \hat{k}_{\lambda},\hat{k}_{\mu} \rangle|)^2\\& = & (|\langle B^{1-\alpha} \hat{k}_{\lambda},A^{\alpha} \hat{k}_{\mu} \rangle|+|\langle B^{\alpha} \hat{k}_{\lambda},A^{1-\alpha} \hat{k}_{\mu} \rangle|)^2\\ & \leq & (\| B^{1-\alpha} \hat{k}_{\lambda}\|\|A^{\alpha} \hat{k}_{\mu}\|+\| B^{\alpha} \hat{k}_{\lambda}\|\|A^{1-\alpha} \hat{k}_{\mu}\|)^2\\&=& \left(\langle A^{2\alpha} \hat{k}_{\mu},\hat{k}_{\mu} \rangle^{1/2}\langle B^{2(1-\alpha)} \hat{k}_{\lambda},\hat{k}_{\lambda} \rangle^{1/2}+\langle A^{2(1-\alpha)} \hat{k}_{\mu},\hat{k}_{\mu} \rangle^{1/2}\langle B^{2\alpha} \hat{k}_{\lambda},\hat{k}_{\lambda} \rangle^{1/2} \right)^2\\ &\leq& 4\left(\frac{\langle A^{2\alpha} \hat{k}_{\mu},\hat{k}_{\mu} \rangle + \langle A^{2(1-\alpha)} \hat{k}_{\mu},\hat{k}_{\mu} \rangle}{2} \right)\left(\frac{\langle B^{2\alpha} \hat{k}_{\lambda},\hat{k}_{\lambda} \rangle+\langle B^{2(1-\alpha)} \hat{k}_{\lambda},\hat{k}_{\lambda} \rangle}{2} \right)\\&& \Big(\mbox{using the inequality $(ab+cd)^2\leq(a^2+c^2)(b^2+d^2)$ for all $a,b,c,d \in \mathbb{R}$}\Big)\\ &\leq& 4\left(\frac{\langle A^{2\alpha} \hat{k}_{\mu},\hat{k}_{\mu} \rangle^r + \langle A^{2(1-\alpha)} \hat{k}_{\mu},\hat{k}_{\mu} \rangle^r}{2} \right)^{1/r}\left(\frac{\langle B^{2\alpha} \hat{k}_{\lambda},\hat{k}_{\lambda} \rangle^s+\langle B^{2(1-\alpha)} \hat{k}_{\lambda},\hat{k}_{\lambda} \rangle^s}{2} \right)^{1/s}\\&&\,\,\,\,\,\,\,\, \,\,\,\,\,\,\,\,\,\, \Big(\mbox{by Lemma \ref{lemma3}}\Big)\\ &\leq& 4\left(\frac{\langle A^{2\alpha r} \hat{k}_{\mu},\hat{k}_{\mu} \rangle + \langle A^{2(1-\alpha)r} \hat{k}_{\mu},\hat{k}_{\mu} \rangle}{2} \right)^{1/r}\left(\frac{\langle B^{2\alpha s} \hat{k}_{\lambda},\hat{k}_{\lambda} \rangle+\langle B^{2(1-\alpha)s} \hat{k}_{\lambda},\hat{k}_{\lambda} \rangle}{2} \right)^{1/s}\\&&\,\,\,\,\,\,\,\,\,\,\,\,\,\,\,\,\,\,\,\, \Big(\mbox{by Lemma \ref{lemma1}}\Big)\\&=&
			4 \left\langle \left(\frac{A^{2\alpha r} +A^{2(1-\alpha)r}}{2}\right) \hat{k}_{\mu},\hat{k}_{\mu}\right \rangle^{1/r}\left\langle \left(\frac{B^{2\alpha s} +B^{2(1-\alpha)s}}{2}\right) \hat{k}_{\lambda},\hat{k}_{\lambda}\right \rangle^{1/s} \\ &\leq& 4 \, \textbf{ber}^{1/r}\left(\frac{A^{2\alpha r} +A^{2(1-\alpha)r}}{2}\right) \textbf{ber}^{1/s}\left(\frac{B^{2\alpha s} +B^{2(1-\alpha)s}}{2}\right).
		\end{eqnarray*}
		Therefore, taking the supremum over all $\lambda,\mu\in\Omega$, we get the desired inequality.
\end{proof}

Considering $s=r$ in Theorem \ref{theo2}, we infer the following corollary.

\begin{cor}\label{cor9}
	Let $A,B\in \mathbb{B}(\mathscr{H})$ be positive. Then 
	\begin{eqnarray}\label{eqn11}
		\left\|A^{\alpha}B^{1-\alpha}+A^{1-\alpha}B^{\alpha}\right\|^{2r}_{ber} \leq 2^{2r-2}\textbf{ber}\left(A^{2\alpha r} +A^{2(1-\alpha)r}\right) \textbf{ber}\left(B^{2\alpha r} +B^{2(1-\alpha)r}\right),
	\end{eqnarray}
	for all $r\geq 1$ and for all $\alpha \in [0,1].$
	
	 In particular, for $\alpha=\frac{1}{2}$ and $r=1$, we have
	\begin{eqnarray}\label{eqn12}
	\left\|A^{1/2}B^{1/2}\right\|_{ber} \leq {\textbf{ber}^{1/2}(A)\textbf{ber}^{1/2}(B)}.
	\end{eqnarray}

	Moreover, if $AB=BA$, then 
	\begin{eqnarray}\label{eqn13}
   \left \|\sqrt{AB} \right\|_{ber} \leq \sqrt{\textbf{ber}(A)\,\textbf{ber}(B)}.
	\end{eqnarray}
\end{cor}

Following Proposition \ref{prop1}, since $\textbf{ber} (\sqrt{AB}) = \| \sqrt{AB} \|_{ber},$  the inequality (\ref{eqn13}) is same as the existing inequality \cite[Cor. 2.10.]{Trd}, namely,
		$\textbf{ber} (\sqrt{AB}) \leq \sqrt{\textbf{ber}(A)\textbf{ber}(B)}$, where $A,B\in \mathbb{B}(\mathscr{H})$ with $A,B \geq 0$ and $AB=BA.$   \\

For  $A,B\in \mathbb{B}(\mathscr{H})$ and $\alpha \in [0,1]$, the  $\alpha$-weighted arithmetic mean of $A$ and $B$ is given by $\alpha A + (1-\alpha)B.$
Now, in the following theorem we obtain an upper bound for the Berezin norm for $\alpha$-weighted arithmetic mean of two opetators.

\begin{theorem}\label{theo3}
	Let $A,B\in \mathbb{B}(\mathscr{H})$. Then
	\begin{eqnarray*}
		\|\alpha A + (1-\alpha)B\|^2_{ber} \leq \textbf{ber} \left(\alpha^2|A|^2+(1-\alpha)^2|B|^2\right)+2\alpha(1-\alpha)\textbf{ber}(B^*A),
	\end{eqnarray*}
	for all $\alpha \in [0,1].$
\end{theorem}

\begin{proof}
	Let $\hat{k}_{\lambda}$ and $\hat{k}_{\mu}$ be two normalized reproducing kernel of $\mathscr{H}.$ Then
	\begin{eqnarray*}
		&&|\langle (\alpha A + (1-\alpha)B)\hat{k}_{\lambda},\hat{k}_{\mu}\rangle|^2 \\& \leq & \|(\alpha A + (1-\alpha)B)\hat{k}_{\lambda}\|^2 \\ &=& \langle (\alpha A + (1-\alpha)B)\hat{k}_{\lambda}, (\alpha A + (1-\alpha)B)\hat{k}_{\lambda} \rangle \\ &=& \alpha^2 \langle A\hat{k}_{\lambda},A\hat{k}_{\lambda} \rangle + (1-\alpha)^2\langle B\hat{k}_{\lambda},B\hat{k}_{\lambda} \rangle + 2\alpha(1-\alpha) \textit{Re} \langle A\hat{k}_{\lambda},B\hat{k}_{\lambda} \rangle \\ & \leq & \alpha^2 \langle |A|^2 \hat{k}_{\lambda},\hat{k}_{\lambda} \rangle + (1-\alpha)^2\langle |B|^2\hat{k}_{\lambda},\hat{k}_{\lambda} \rangle + 2\alpha(1-\alpha) |\langle B^*A\hat{k}_{\lambda},\hat{k}_{\lambda} \rangle|\\ &\leq & \textbf{ber}\left(\alpha^2|A|^2+(1-\alpha)^2|B|^2\right)+2\alpha(1-\alpha)\textbf{ber}(B^*A).
	\end{eqnarray*}
	Therefore, taking the supremum over all $\lambda,\mu\in\Omega$, we get 
	\begin{eqnarray*}
		\|\alpha A + (1-\alpha)B\|^2_{ber} \leq \textbf{ber}\left (\alpha^2|A|^2+(1-\alpha)^2|B|^2\right)+2\alpha(1-\alpha)\textbf{ber}(B^*A).
	\end{eqnarray*}
\end{proof}

Putting $\alpha=\frac{1}{2}$ in Theorem \ref{theo3}, we get the following corollary which presents upper bound for the Berezin norm of the sum of two operators.

\begin{cor}
		Let $A,B\in \mathbb{B}(\mathscr{H})$. Then
	\begin{eqnarray*}
		\| A + B\|^2_{ber} \leq \textbf{ber}\left (|A|^2+|B|^2\right)+2 \, \textbf{ber}(B^*A).
	\end{eqnarray*}
\end{cor}

\begin{remark}
The following inequalities for the sum of the product of operators defined on a complex Hilbert space can be obtained using analogous argument as described in Theorems \ref{theo1},  \ref{theo2} and \ref{theo3}.\\
(i) Let $A,B,C,D,X,Y\in \mathbb{B}(\mathbb{H})$, and let $\alpha \in [0,1].$ Then
	\begin{align*}
	&\left\|\frac{A^*XB+C^*YD}{2}\right\|^{2} \\& \leq  \left\| \frac{(B^*|X|^{2\alpha }B)^r+(D^*|Y|^{2\alpha }D)^r}{2}\right\|^{1/r} \left\|\frac{(A^*|X^*|^{2(1-\alpha)}A)^s +(C^*|Y^*|^{2(1-\alpha)}C)^s}{2}\right\|^{1/s},
	\end{align*}
	for all $r,s \geq 1.$ \\
(ii) Let $A,B,C,D\in \mathbb{B}(\mathbb{H})$. Then
 	\begin{align*}\label{eq22}
	\left\|\frac{A^*B+C^*D}{2}\right\|^{2} \leq \left\| \frac{(B^*B)^r+(D^*D)^r}{2}\right\|^{1/r}\left\|\frac{(A^*A)^s +(C^*C)^s}{2}\right\|^{1/s}, 
	\end{align*} 
for all $r,s\geq 1$.\\ 
(iii) Let $A,B\in \mathbb{B}(\mathbb{H})$ be positive. Then 
	\begin{eqnarray*}
		\left\|\frac{A^{\alpha}B^{1-\alpha}+A^{1-\alpha}B^{\alpha}}{2}\right\|^{2} \leq \left \|\frac{A^{2\alpha r} +A^{2(1-\alpha)r}}{2}\right \|^{1/r} \left \| \frac{B^{2\alpha s} +B^{2(1-\alpha)s}}{2}\right \|^{1/s},
	\end{eqnarray*}
	for all $r,s \geq 1$ and for all $\alpha \in [0,1].$\\
(iv) Let $A,B\in \mathbb{B}(\mathbb{H})$. Then
	\begin{eqnarray*}
		\|\alpha A + (1-\alpha)B\|^2 \leq  \left \|\alpha^2|A|^2+(1-\alpha)^2|B|^2\right \|+2\alpha(1-\alpha)w(B^*A),
	\end{eqnarray*}
	for all $\alpha \in [0,1].$

Note that the inequality obtained in (ii) was developed independently in \cite[Th. 3.]{drag}.

\end{remark}

\noindent \textbf{Declarations.}\\
Authors declare that data sharing is not applicable to this article as no datasets were generated or analysed during the current study.

\bibliographystyle{amsplain}

\end{document}